\documentclass[12pt]{amsart}

\font\de=cmssi12 \font\dediez=cmssi10
\usepackage{amsmath, amssymb}
\usepackage{amsthm}
\usepackage{amsfonts}
\usepackage{graphicx}
\usepackage{psfrag}

\newtheorem{teo}{\bf Theorem}[section]

\newtheorem{lema}[teo]{\bf Lemma}
\newtheorem{prop}[teo]{\bf Proposition}

\newtheorem{defi}[teo]{\bf Definition}

\newcommand{\s}{\sigma}
\newcommand{\eps}{\varepsilon}

\newcommand{\st}{{\rm st}}
\newcommand{\T}{\mathbb T}
\begin{document}
\title[]{Quasi-Anosov diffeomorphisms of 3-manifolds}
\author{T. Fisher}
\address{Brigham Young University, Dept. of Math., 292 TMCB, Provo, UT 84602.}
\email{tfisher@math.byu.edu}

\author{M. Rodriguez Hertz}
\address{IMERL\\ Facultad de Ingenier\'{\i}a\\ Julio Herrera y Reissig 565\\ 11300 Montevideo - Uruguay}
\email{jana@fing.edu.uy}

\date{\today}
\thanks{Partially supported by NSF Grant \#DMS0240049, Fondo Clemente Estable 9021 and PDT}
\subjclass[2000]{37D05; 37D20}
\keywords{Dynamical Systems, Hyperbolic Set, Robustly Expansive, Quasi-Anosov}
\begin{abstract}
 In 1969, Hirsch posed the following problem: given a
diffeomorphism $f:N\to N$, and a compact invariant hyperbolic set
$\Lambda$ of $f$, describe the topology of $\Lambda$ and the
dynamics of $f$ restricted to $\Lambda$. We solve the problem where
$\Lambda=M^3$ is a closed $3$-manifold: if $M^3$ is orientable, then
it is a connected sum of tori and handles; otherwise it is a
connected sum of tori and handles quotiented by involutions.\par The
dynamics of the diffeomorphisms restricted to $M^3$, called {\dediez
quasi-Anosov diffeomorphisms}, is also classified: it is the
connected sum of DA-diffeomorphisms, quotiented by
commuting involutions.
\end{abstract}
\maketitle \thispagestyle{empty}
\section{Introduction}
This paper deals with hyperbolic sub-dynamics. It is related to a
problem posed by M. Hirsch, around 1969: given a diffeomorphism
$f:N\to N$, and a compact invariant hyperbolic set $\Lambda$ of $f$,
describe the topology of $\Lambda$ and the dynamics of $f$
restricted to $\Lambda$. Hirsch asked, in particular, whether the
fact that $\Lambda$ were a manifold $M$ would imply that the
restriction of $f$ to $M$ is an Anosov diffeomorphism \cite{hirsch}.
However, in 1976, Franks and Robinson gave an example of a
non-Anosov hyperbolic sub-dynamics in the connected sum of two
${\mathbb T}^3$ \cite{frob} (see below). There are also examples of
hyperbolic sub-dynamics in non-orientable 3-manifolds, for instance,
the example of Zhuzhoma and Medvedev \cite{medvedevzhuzhoma}. We
show here that all examples of 3-manifolds that are hyperbolic
invariant sets are, in fact, finite connected sums of the examples
above and handles $S^2\times S^1$ (see definitions in
\S\ref{section.qad} and \S\ref{section thm1.1})
\begin{teo}\label{thm.kneser.decomposition} Let
$f:N\to N$ be a diffeomorphism, and let $M\subset N$ be a hyperbolic
invariant set for $f$ such that $M$ is a closed orientable
3-manifold. Then the Kneser-Milnor prime decomposition of $M$ is
$$M=T_1\#\dots\# T_k\# H_1\#\dots \# H_r$$
the connected sum of $k\geq 1$ tori $T_i=\T^3$ and $r\geq 0$ handles
$H_j=S^2\times S^1$. In case $M$ is non-orientable, then $M$
decomposes as
$$M=\tilde{T_1}\#\dots\#\tilde{T_k}\# H_1\#\dots\# H_r$$
the connected sum of $k\geq 1$ tori quotiented by involutions
$\tilde{T_i}=\T^3|\theta_i$ and $r$ handles $H_j=S^2\times S^1$.
\end{teo}
In 1976, Ma\~n\'e obtained the following characterization
\cite{manhe} (see also Theorem \ref{teo.manhe}): $g:M\to M$ is the
restriction of another diffeomorphism to a hyperbolic set $M$ that
is a closed manifold, if and only if $g$ is {\de quasi-Anosov}; that
is, if it satisfies Axiom A and all intersections of stable and
unstable manifolds are quasi-transversal, i.e.:
\begin{equation}\label{eq.quasi-transversal} T_xW^s(x)\cap
T_xW^u(x)=\{0\}\qquad\forall x\in M
\end{equation}\par
 The Franks-Robinson's
example of a non-Anosov quasi-Anosov diffeomorphism is essentially
as follows: they consider a hyperbolic linear automorphism of a
torus $T_1$ with only one fixed point, and its inverse in another
torus $T_2$. They produce appropriate deformations on each torus
(DA-diffeomorphisms) around their respective fixed points. Then they
cut suitable neighborhoods containing these fixed points, and
carefully glue together along their boundary so that the stable and
unstable foliations intersect quasi-transversally. This is a
quasi-Anosov diffeomorphism in the connected sum of $T_1$ and $T_2$,
and hence $T_1\# T_2$ is a compact invariant hyperbolic set of some
diffeomorphism. The non-orientable example by Medvedev and Zhuzhoma
\cite{medvedevzhuzhoma} is similar to Franks and Robinson's, but
they perform a quotient of each $T_i$ by an involution before gluing
them together.
\par%
The second part of this work, a classification of the dynamics of
quasi-Anosov diffeomorphisms of $3$-manifolds, shows that all
examples are, in fact, connected sums of the basic examples above:
\begin{teo}\label{teo.quasi.Anosov.dim3} Let $g:M\to M$ be a quasi-Anosov
diffeomorphism of a closed 3-manifold $M$. Then
\begin{enumerate}
\item \label{item1} The non-wandering set $\Omega(g)$ of $g$ consists
of a finite number of codimension-one expanding attractors,
codimension-one shrinking repellers and hyperbolic periodic points.
\item\label{item2} For each attractor $\Lambda$ in $\Omega(g)$, there exist a hyperbolic
toral automorphism $A$ with stable index one, a finite set $Q$ of
$A$-periodic points, and a linear involution $\theta$ of ${\mathbb
T}^3$ fixing $Q$ such that the restriction of $g$ to its basin of
attraction $W^s(\Lambda)$ is topologically conjugate to a
DA-diffeomorphism $f^A_Q$ on the punctured torus ${\mathbb T}^3-Q$
quotiented by $\theta$. In case $M$ is an orientable manifold,
$\theta$ is the identity map. An analogous result holds for the
repellers of $\Omega(g)$.
\end{enumerate}
\end{teo}
Item (\ref{item2}) above is actually a consequence of item
(\ref{item1}), as it was shown by Plykin in \cite{plykin1,plykin2},
see also \cite{grines1977} and \cite{GZ1979}. A statement of the
result can be found in Theorem \ref{teo.plykin} in this work. The proof
of Theorem \ref{teo.quasi.Anosov.dim3} is in \S\ref{section.thm1.2}.
Theorem \ref{teo.quasi.Anosov.dim3}, in fact, implies Theorem
\ref{thm.kneser.decomposition}. This is proved in \S \ref{section
thm1.1}.
\par
Let us see how a handle $S^2\times S^1$ could appear in the prime
decomposition of $M$: Consider a linear automorphism of a torus
$T_1$, and its inverse in a torus $T_2$, as in Franks-Robinson's
example. Then, instead of exploding a fixed point, one explodes and
cuts around an orbit of period 2 in $T_1$ and in $T_2$. The rest of
the construction is very similar, gluing carefully as in that
example to obtain a quasi-Anosov dynamics. This gives the connected
sum of two tori and a handle. Explanation and details can be found in \S\ref{section thm1.1}.\par%
Let us also mention that in a previous work \cite{rhuv} it was shown
there exist a codimension-one expanding attractor and a
codimension-one shrinking repeller if $g$ is a quasi-Anosov
diffeomorphism of a $3$-manifold that is not Anosov. The fact that
only ${\mathbb T}^3$ can be an invariant subset of any known Anosov
system was already shown by A. Zeghib \cite{zeghib}. In that case,
the dynamics is Anosov. See also
\cite{franks} and \cite{manhe2}.\par%
This work is also related to a work by Grines and Zhuzhoma
\cite{GZ}. There they prove that if an $n$-manifold supports a
structurally stable diffeomorphism with a codimension-one expanding
attractor, then it is homotopy equivalent to $\T^n$, and
homeomorphic to $\T^n$ if $n\ne4$. In a certain sense, the results
deal with complementary extreme situations in the Axiom A world:
Grines-Zhuzhoma result deals with structurally stable
diffeomorphisms, which are Axiom A satisfying the strong
transversality condition. This means that all $x,y$ in the
non-wandering set satisfy at their points $z$ of intersection: $T_z
W^s(x)\pitchfork T_z W^u(y)$. In particular,
$$\dim E^s_x+\dim E^u_y \geq n$$
In our case, we deal with quasi-Anosov diffeomorphisms, which are
Axiom A satisfying equality (\ref{eq.quasi-transversal}). In
particular, for $x,y,z$ as above, we have $T_zW^s(x)\cap
T_zW^u(y)=\{0\}$, so:
$$\dim E^s_x+\dim E^u_y \leq n$$
In the intersection of both situations are, naturally, the Anosov
diffeomorphisms.\par
 Observe that it makes sense to get a
classification of the dynamical behavior of quasi-Anosov on its
non-wandering set, since quasi-Anosov are $\Omega$-stable
\cite{manhe} (see also \S \ref{section.qad}). They form an open set,
due to quasi-transversality condition (\ref{eq.quasi-transversal}).
Moreover, they are the $C^1$-interior of expansive diffeomorphisms,
that is, they are robustly expansive \cite{manhe3}. However,
3-dimensional quasi-Anosov diffeomorphisms of $M\ne \T^3$ are never
structurally stable, so they are approximated by other quasi-Anosov
diffeomorphisms with different dynamical behavior, but similar
asymptotic behavior (Proposition \ref{approximation.qad}).\par
Finally, in Section \ref{section.partially.hyperbolic} we study
quasi-Anosov diffeomorphisms in the presence of partial
hyperbolicity (see definitions in
\S\ref{section.partially.hyperbolic}). We obtain the following
result under mild assumptions on dynamical coherence:
\begin{teo}\label{ph.qad} If $f:M^3\to M^3$ is a quasi-Anosov
diffeomorphism that is partially hyperbolic, and either $E^{cs}$ or
$E^{cu}$ integrate to a foliation, then $f$ is Anosov.
\end{teo}

\setcounter{section}{1}
{\it Acknowledgments.} We want to thank the referees for valuable
comments, and Ra\'ul Ures for suggestions. The second author is
grateful to the Department of Mathematics of the University of
Toronto, and specially to Mike Shub, for kind hospitality.

\section{Basic definitions}
%
Let us recall some basic definitions and facts: Given a
diffeomorphism $f:N\to N$, a compact invariant set $\Lambda$ is a
{\de hyperbolic set} for $f$ if there is a $Tf$-invariant splitting
of $TN$ on $\Lambda$: $$T_x N=E^s_x\oplus E^u_x\qquad\forall
x\in\Lambda$$ such that all unit vectors $v^\sigma \in
E^\sigma_\Lambda$, with $\sigma=s,u$ satisfy
$$|Tf(x)v^s|<1< |Tf(x)v^u|$$ for some suitable Riemannian metric
$|.|$. The {\de non-wandering set} of a diffeomorphism $g:M\to M$ is
denoted by $\Omega (g)$ and consists of the points $x\in M$, such
that for each neighborhood $U$ of $x$, the family
$\{g^n(U)\}_{n\in{\mathbb Z}}$ is not pairwise disjoint. The
diffeomorphism $g:M\to M$ satisfies {\de Axiom A} if $\Omega (g)$ is
a hyperbolic set for $g$ and periodic points are dense in
$\Omega(g)$. The {\de stable manifold} of a point $x$ is the set
$$W^s(x)=\{y\in M: d(f^n(x),f^n(y))\to 0\quad \mbox{if}\quad n\to
\infty\}$$ where $d(.,.)$ is the induced metric; the {\de unstable
manifold} $W^u(x)$ is defined analogously for $n\to -\infty$. If $g$
satisfies Axiom A, then $W^s(x)$ and $W^u(x)$ are immersed manifolds
for each $x\in M$ (see for instance \cite{shub}). Also, if
$d_\sigma$ is the intrinsic metric of the invariant manifold
$W^\sigma(x)$, for $\sigma=s,u$, one has constants $C,\eps>0$ and
$0<\lambda<1$ such that, for instance, if $y\in W^s(x)$, and $d_s(x,y)\leq\eps$  for some small $\eps>0$ then %
\begin{equation}\label{hyperbolic.metric}
d_s(f^n(x),f^n(y))\leq C\lambda^n d_s(x,y)\qquad \forall n\geq0
\end{equation}
an analogous bound holds for the unstable manifold. \par
Due to the Spectral Decomposition Theorem of Smale \cite{smale}, if
$g$ is Axiom A, then $\Omega(g)$ can be decomposed into  disjoint
compact invariant sets, called {\de basic sets}:
$$\Omega(g)=\Lambda_1\cup\dots\cup\Lambda_r,$$
each $\Lambda_i$ contains a dense orbit. Furthermore, each
$\Lambda_i$ can be decomposed into disjoint compact sets
$\Lambda_i=\Lambda_{i,1}\cup\dots\cup\Lambda_{i,k}$ such that there
exists an $n\in\mathbb{N}$ where each $\Lambda_{i,j}$ is invariant
and topologically mixing for $g^n$.  A set $X$ is
{\de topologically mixing} for a diffeomorphism $f$ if for each pair of
nonempty open sets $U$ and $V$ of $X$, there is $K>0$ such that
$$f^{k}(U)\cap V\ne
\emptyset\qquad\forall k\geq K.$$\par%
Note that $\dim E^s_x$ is constant for
$x$ varying on a basic set $\Lambda$,
 we shall call this amount the {\de stable index} of $\Lambda$,
 and will denote it by $\st(\Lambda)$. \par
For any set $\Lambda\subset M$, let us denote by $W^\sigma(\Lambda)$
the set $\bigcup_{x\in \Lambda}W^\s(x)$, where $\sigma=s,u$. We
define the
following (reflexive) relation among basic sets:%
$$\Lambda_1\to\Lambda_2\qquad\iff\qquad W^u(\Lambda_1)\cap
W^s(\Lambda_2)\ne\emptyset$$%
The relation $\to$ naturally extends to a transitive relation
$\succeq$: $$\Lambda_i\succeq\Lambda_j\quad\iff\quad
\Lambda_i\to\Lambda_{k_1}\to\dots\to\Lambda_{k_r}\to\Lambda_j$$
where $\Lambda_{k_1},\dots,\Lambda_{k_r}$ is a finite sequence of
basic sets. The diffeomorphism satisfies the {\de no-cycles
condition} if $\succeq$ is anti-symmetric: $$\Lambda_1\succeq
\Lambda_2\quad\mbox{and}\quad \Lambda_2\succeq \Lambda_1\quad
\implies \quad \Lambda_1=\Lambda_2$$ In this case $\succeq$
defines a partial order among basic sets.\par%
 We shall call $\Lambda$ an {\de attractor} if $\Lambda$ is a basic set such that $W^u(\Lambda)=\Lambda$.
 Note that this implies that there exists a neighborhood $U$ of $\Lambda$ such that $\Lambda=\bigcap_{n\in\mathbb{N}}f^n(U)$.
 Similarly, we shall call $\Lambda$ a {\de repeller} if $\Lambda$ is a basic set such that $W^s(\Lambda)=\Lambda$. If $g$
is Axiom A and satisfies the no-cycles
condition, then hyperbolic attractors and repellers are, respectively, the minimal and maximal elements of $\succeq$.\par%
A hyperbolic attractor $\Lambda$ is a {\de codimension-one expanding
attractor} if all $x\in \Lambda$ satisfy $\dim W^u(x)=\dim M-1$.
Codimension-one shrinking repellers are defined analogously.\par
Note that an attractor can have topological dimension $\dim M-1$,
and still be not expanding. See the survey \cite{GZ2006} on
expanding attractors for a discussion on this topic.

\section{Quasi-Anosov diffeomorphisms}\label{section.qad}
Let $f:N\to N$ be a diffeomorphism of a Riemannian manifold.
\begin{defi}\label{def.quasitransversal}
The sets $W^s(x)$ and $W^u(x)$ have a point of {\de
quasi-transversal} intersection at $x$ if
$$T_xW^s(x)\cap T_xW^u(x)=\{0\} $$ (see figure \ref{fig quasi
transv})
\end{defi}
\begin{figure}[h]
\psfrag{x}{$y$}\psfrag{y}{$z$}\psfrag{z}{$x$}\psfrag{wsy}{$W^s(x)$}
\psfrag{wsx}{$W^s(y)$}\psfrag{wux}{$W^u(x)$}\psfrag{wuy}{$W^u(z)$}
\includegraphics[height=5cm]{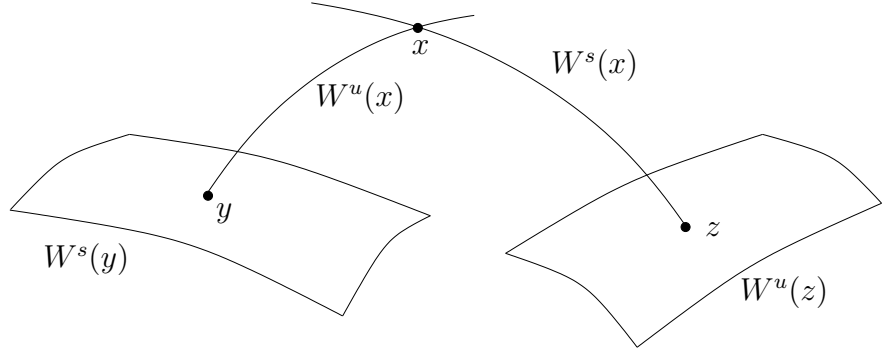}
\caption{\label{fig  quasi transv} Quasi-transversal intersection at
$x$}
\end{figure}
At a point of quasi-transversal intersection $x$, all vectors in
$E_x^s$ form a positive angle with vectors in $E^u_x$. But this does
not necessarily imply transversality, as can be seen in
Figure~\ref{fig quasi transv}.\par %
Let us note the difference between this definition and the strong
transversality condition. There, transversality is required at the
intersection points of $W^u(x)$ and $W^s(y)$, but this can be
attained without quasi-transversality, for instance, if we had two
planes intersecting at a curve in a 3-dimensional setting.

Observe that a structurally stable quasi-Anosov diffeomorphism is
Anosov (see \cite{manhe} and references therein). On the other hand,
quasi-Anosov diffeomorphisms satisfy the no-cycles condition (see
below), and hence they are $\Omega$-stable. Also, quasi-Anosov are a
$C^1$-open set od diffeomorphisms\cite{manhe3}. 
\begin{prop}\label{approximation.qad}
A quasi-Anosov diffeomorphism $f$ that is not Anosov is approximated
by $\Omega$-conjugate quasi-Anosov diffeomorphisms that are not
topologically conjugate to $f$.
\end{prop}

 The following theorem by Ma\~n\'e relates quasi-Anosov
diffeomorphisms with hyperbolic sub-dynamics.
\begin{teo}[Ma\~n\'e \cite{manhe}] \label{teo.manhe}
A diffeomorphism $g$ is a quasi-Anosov diffeomorphism if and only if $M$ can be
embedded as a hyperbolic set for a diffeomorphism $f:N\to N$ by
means of an embedding $i:M\hookrightarrow N$ satisfying $fi=ig$.
\end{teo}
This characterization reduces the proof of Theorem
\ref{thm.kneser.decomposition} to proving Theorem
\ref{teo.quasi.Anosov.dim3}. See also \S \ref{section thm1.1}. We
shall review some
properties of quasi-Anosov diffeomorphisms:%
\begin{prop}\cite{manhe}
Quasi-Anosov diffeomorphisms satisfy the no-cycles condition.
\end{prop}
\begin{proof}If $\Lambda_i$ and $\Lambda_j$ are two basic sets satisfying
$\Lambda_i\to\Lambda_j$, then $W^u(x_i)\cap W^s(x_j)\ne\emptyset$
for some $x_k\in\Lambda_k$. It follows from quasi-transversality
that
$$[n-\st(\Lambda_i)]+\st(\Lambda_j)=\dim E^u_{x_i}+\dim
E^s_{x_j}\leq n$$ where $n$ is the dimension of $M$, hence
$\st(\Lambda_j)\leq \st(\Lambda_i)$. We get by transitivity that
\begin{equation}\label{estable.decrece} \Lambda_i\succeq
\Lambda_j\quad\Rightarrow\quad \st(\Lambda_i)\geq
\st(\Lambda_j)\end{equation} Suppose that, 
$$\Lambda_1\to\Lambda_2\to\dots\to\Lambda_k\to\Lambda_1.$$
We have, in the first place that $\st(\Lambda_i)=\st(\Lambda_1)$,
hence all intersections $x_i\in W^u(\Lambda_i)\cap
W^s(\Lambda_{i+1})$ for $i=1,\dots,k-1$, and $x_k\in
W^u(\Lambda_k)\cap W^s(\Lambda_1)$ are transversal, since 
$$[n-\st(\Lambda_i)]+\st(\Lambda_1)=\dim E^u_{x_i}+\dim
E^s_{x_1}= n$$. This implies the
$x_i$'s belong to $\Omega(g)$, hence $\Lambda_i=\Lambda_i$ for all
$i=1,\dots, k$, and so $g$ satisfies the no-cycles
condition.\end{proof}

In the particular case of a quasi-Anosov diffeomorphism of a
3-dimensional manifold, this implies there can be only basic sets
with stable index 2 or 1; and basic sets with stable index one can
only succeed basic sets with stable index one. We delay the proof of
the next proposition until the next section.
\begin{prop}\label{p.basic}
If $f$ is a quasi-Anosov diffeomorphism and $\Lambda_0$ is a
co\-dimension-one expanding attractor, and $\Lambda$ is a basic set
satisfying $\Lambda\succeq \Lambda_0$ with $\st(\Lambda)=1$, then
$\Lambda=\Lambda_0$ or else $\Lambda$ is a periodic point.
\end{prop}
Analogously, if $\Lambda_0$ is a codimension-one repeller, and
$\Lambda$ is a basic set satisfying $\Lambda_0\succeq\Lambda$ with
$\st(\Lambda)=2$, then $\Lambda=\Lambda_0$ or $\Lambda$ is a
periodic point. This implies:
\begin{prop}\label{codim.one.att}
All attractors of a quasi-Anosov diffeomorphism of a 3-dimensional
manifold are codimension-one expanding attractors, unless the
diffeomorphism is Anosov. An analogous statement holds for
repellers.
\end{prop}
\begin{proof}
Indeed, let $\Lambda_R$ be a repeller such that $\st(\Lambda_R)=1$
(hence, not co\-di\-men\-sion-one). There is a maximal chain of
$\succeq$ containing $\Lambda_R$. Let $\Lambda_A$ be a minimal
element of that chain. Then, due to (\ref{estable.decrece}) in the
proof above, $\Lambda_A$ is a codimension-one expanding attractor.
But then Proposition \ref{p.basic} implies the repeller $\Lambda_R$
equals $\Lambda_A$, since $\Lambda_A$ cannot be a periodic point.
Therefore $\Lambda_R=\Lambda_A=M$.
\end{proof}
Note that Propositions \ref{p.basic} and \ref{codim.one.att} above
prove Item (1) of Theorem \ref{teo.quasi.Anosov.dim3}. Item (2) of
Theorem \ref{teo.quasi.Anosov.dim3} follows from results in next
section.

\section{Codimension-one expanding attractors and shrinking
repellers - Proof of Theorem
\ref{teo.quasi.Anosov.dim3}}\label{section.thm1.2}
Before proving Proposition~\ref{p.basic} we review properties of
codimension-one expanding attractors.  A codimension-one expanding
attractor $\Lambda$ is {\de orientable} if the intersection index of
$W^s(x)\cap W^u(y)$ is constant at all its intersection points, for
$x,y\in \Lambda$. This notion was first introduced by Grines \cite{grines1974,grines1975}. %
Let us also recall the following result by Zhuzhoma and Medvedev:
\begin{teo}[Medvedev-Zhuzhoma \cite{medvedevzhuzhoma}]\label{thm.medvedev.zhuzhoma} If $M$ is an orientable closed 3-manifold, then
all codimension-one expanding attractors and shrinking repellers are
orientable.
\end{teo}
Derived from Anosov (or DA-) diffeomorphisms were introduced by
Smale in \cite{smale} (see also \cite{williams}). They are certain
 deformations of hyperbolic automorphisms of the torus. We shall
 use the following definition \cite{plykin1}:\par
 Corresponding to a hyperbolic toral automorphism $A$ with stable index one, and a finite
 set $Q$ of $A$-periodic points, there is a diffeomorphism
 $f^A_Q:{\mathbb T}^3\to {\mathbb T}^3$ diffeotopic to $A$, such that $\Omega(f^A_Q)=\Lambda\cup Q$,
 where $\Lambda$ is a codimension-one expanding attractor and $Q$ is a finite set of $f^A_Q$-repelling periodic points.
 The stable manifolds of $f^A_Q$ coincide with the stable manifolds
 of $A$, except for a finite set of lines ${\mathcal L}_Q$. Each
 line $L\in{\mathcal L}_Q$ contains a point $q\in Q$. The component
 of $L-\Lambda$ containing $q$ is an interval whose endpoints
 $p^\pm$ are periodic boundary points of $\Lambda$.
We call $f^A_Q$  a {\de DA-diffeomorphism}.\par Plykin obtains
models for connected codimension-one expanding attractors using
DA-diffeomorphisms \cite{plykin1,plykin2}. See also \S8 of
\cite{GZ}. We shall also use some of his intermediate results:
\begin{teo}[Plykin \cite{plykin1}]\label{homotopy.torus} If $\Lambda$ is a connected
orientable codimension-one expanding attractor of a diffeomorphism
$g:M^3\to M^3$, then $W^s(\Lambda)$ has the homotopy type of
${\mathbb T}^3-Q$, where $Q$ is a finite set of points. There is a
finite point-compactification $\overline{W^s(\Lambda)}$ of
$W^s(\Lambda)$ having the homotopy type of ${\mathbb T}^3$, and a
homeomorphism $\bar g:\overline{W^s(\Lambda)}\to
\overline{W^s(\Lambda)}$ extending $g|_{W^s(\Lambda)}$, and
admitting two $\bar g$-invariant fibrations that extend,
respectively, the stable and unstable manifolds of $\Lambda$.
\end{teo}
An analogous result holds for non-orientable attractors: there
exists a two-sheeted covering $\pi:\overline{W^s(\Lambda)}\to
W^s(\Lambda)$ and a covering homeomorphism $\bar g:
\overline{W^s(\Lambda)}\to \overline{W^s(\Lambda)}$ that commutes
with the involution $\theta: \overline{W^s(\Lambda)}\to
\overline{W^s(\Lambda)}$ associated to $\pi$, such that
$\overline{W^s(\Lambda)}$ has the homotopy type of ${\mathbb T}^3$
\cite{plykin2}.\par Let us note that results above do not require
that $\Lambda$ has a dense orbit.
\begin{teo}[Plykin \cite{plykin1,plykin2}]\label{teo.plykin}
If $\Lambda$ is a connected orientable codimension-one expanding
attractor of a diffeomorphism $g:M^n\to M^n$ having a dense unstable
manifold, then there exist a hyperbolic toral automorphism $A$ with
stable index one, and a finite set $Q$ of $A$-periodic points, such
that $g|_{W^s(\Lambda)}$ is topologically conjugate to the
DA-diffeomorphism
$f^A_Q|_{{\mathbb T}^n-Q}$.\par%
If $\Lambda$ is non-orientable, then there is a two-sheeted covering
$\pi:\overline{W^s(\Lambda)}\to W^s(\Lambda)$ with an associate
involution $\theta:\overline{W^s(\Lambda)}\to
\overline{W^s(\Lambda)}$, and a covering homeomorphism
$g:\overline{W^s(\Lambda)}\to\overline{W^s(\Lambda)}$ commuting with
$\theta$ that is topologically conjugate to a DA-diffeomorphism
$f^A_Q$ as described above.
\end{teo}

Let us note that, in case where the manifold $M$ is a torus, this result
was obtained in \cite{GZ1979}, see also \cite{grines1977} for the
two-dimensional case.\par Next, we state some of the results
obtained in~\cite{GZ} and follow the general outline and notation.

Let $\Lambda$ be a codimension-one expanding attractor.  We will
assume for now that $\Lambda$ is orientable.  (The non-orientable
case will follow by taking a double cover and looking at the
orientable case). A point $p$ is a {\de boundary point} of a
codimension-one expanding attractor $\Lambda$ if there exists a
connected component of $W^s(p)- p$, denoted $W^s_{\emptyset}(p)$,
not intersecting $\Lambda$. Boundary points for hyperbolic
codimension-one expanding attractors are finite and periodic
\cite{plykin1}. For $z\in\Lambda$ and given points $x,y\in W^s(z)$
we denote $(x,y)^s$ ( respectively $[x,y]^s$) the open (closed) arc
of $W^s(z)$ with endpoints $x$ and $y$.  If $p$ is a boundary point
of $\Lambda$ and $x\in W^u(p)-p$, then there is a unique arc
$(x,y)^s_{\emptyset}$ such that $(x,y)^s\cap\Lambda=\emptyset$ and
$y\in\Lambda$.  If $z\in W^s(\Lambda)-\Lambda$, then either
$z\in(x,y)_{\emptyset}^s$ for some $x$ and $y$ elements of the
unstable manifolds of boundary points, or $z\in W^s_{\emptyset}(p)$
for some boundary point $p\in\Lambda$.

The boundary points $p_1$ and $p_2$ are called {\de associated} if
for each point $x\in W^u(p_1)$ there exists an arc
$(x,y)^s_{\emptyset}$ where $y\in W^u(p_2)$, and similarly for each
point $y\in W^u(p_2)$ there is an arc $(x,y)^s_{\emptyset}$ where
$x\in W^u(p_1)$. The boundary point $p_1$ is said to be {\de paired}
if there exists a boundary point $p_2$ such that $p_1$ and $p_2$ are
associated \footnote{This concept also appears as {\de 2- bunched}
in the bibliography. It is not to be confused with the concept of
{\it center bunching} used for partially hyperbolic
diffeomorphisms.}. Two associated boundary points have always the
same period $m$. If $\mathrm{dim}(M)\geq 3$, then all boundary
points are paired.

For associated periodic points $p_1$ and $p_2$ let
$$\varphi_{p_1,p_2}:(W^u(p_1)-p_1)\cup(W^u(p_2)-p_2)\rightarrow (W^u(p_1)-p_1)\cup(W^u(p_2)-p_2)$$ be defined by $\varphi_{p_1,p_2}(x)=y$ whenever $(x,y)_{\emptyset}^s$.  The continuous dependence of stable and unstable manifolds implies that $\varphi_{p_1,p_2}$ is a homeomorphism.  We may naturally extend $\varphi_{p_1,p_2}$ to be a homeomorphism  of $W^u(p_1)\cup W^u(p_2)$ to itself by defining $\varphi_{p_1,p_2}(p_1)=p_2$ and $\varphi_{p_1,p_2}(p_2)=p_1$.

Fix $D_{p_1}$ a closed disk in $W^u(p_1)$ containing $p_1$ in the interior such that $D_{p_1}\subset\mathrm{int}(f^m(D_{p_1}))$.  The boundary of $D_{p_1}$ is a circle denoted $S_{p_1}$.  The circles $S_{p_1}$ and $f^m(S_{p_1})$ bound an annulus contained in $W^u(p_1)$ denoted $A_{p_1}$.

Since $\varphi_{p_1,p_2}$ is a homeomorphism we can define
\begin{itemize}
\item a closed disk $D_{p_2}=\varphi_{p_1,p_2}(D_{p_1})$ in $W^u(p_2)$,
\item a circle $S_{p_2}=\varphi_{p_1,p_2}(S_{p_1})$, and
\item an annulus $A_{p_2}=\varphi_{p_1,p_2}(A_{p_1})$.
\end{itemize}
The set
$$C_{p_1,p_2}=\bigcup_{x\in S_{p_1}}(x,\varphi_{p_1,p_2}(x))_{\emptyset}^s$$ is called a {\de connecting cylinder of $p_1$ and $p_2$} and is homeomorphic to the open $2$-cylinder $S^1\times(0,1)$.  The set
$$S_{p_1,p_2}=D_{p_1}\cup D_{p_2}\cup C_{p_1,p_2}$$ is called a {\de characteristic sphere for $p_1$ and $p_2$} and is homeomorphic to a sphere.

Define
$$A_{p_1,p_2}=\bigcup_{x\in A_{p_1}}[x,\varphi_{p_1,p_2}(x)]^s_{\emptyset}$$ which is homeomorphic to an annulus times an interval.  Let
$$D_{p_1,p_2}=\bigcup_{j\geq 0}f^{jm}(A_{p_1,p_2})=\bigcup_{x\in W^u(p_1)-\mathrm{int}D_{p_1}}[x,\varphi_{p_1,p_2}(x)]^s_{\emptyset}$$ and denote $\pi_{p_1}$ as the projection from $D_{p_1,p_2}$ to $W^u(p_1)-\mathrm{int}(D_{p_1})$.  Then the triple $(D_{p_1,p_2},W^u(p_1)-\mathrm{int}(D_{p_1}),\pi_{p_1})$ is a trivial fiber bundle with fiber the interval $[0,1]$.

The following is Corollary 3.1 in~\cite{GZ}.

\begin{lema}\label{l.intersect} Let $\Lambda_0$ be a codimension-one orientable expanding attractor and $p_1,p_2$
are associated boundary points on $\Lambda_0$.  Suppose $\Lambda$ is
another basic set of $M$ for $f$ with $\st(\Lambda)=1$.  If there
exists a point $z\in\Lambda$ such that $W^u(z)\cap
D_{p_1,p_2}\neq\emptyset$, then $W^u(z)$ intersects $C_{p_1,p_2}$.
\end{lema}

Theorem 6.1 in~\cite{GZ} is similar to the following lemma.

\begin{lema} Suppose $f$ is a quasi-Anosov diffeomorphism of a closed 3-manifold $M$ and $\Lambda_0$ is an orientable codimension-one expanding
attractor of $f$.  Let $\Lambda\neq\Lambda_0$ be a basic set of $\mathrm{st}(\Lambda)=1$ such that $W^u(\Lambda)\cap D_{p_1,p_2}\neq\emptyset$.  Let $C\subset D_{p_1,p_2}\cap W^u(z)$ be a component of the intersection of $D_{p_1,p_2}\cap W^u(z)$, where $z\in\Lambda$ is a periodic point.  Then $W^u(z)\cap C_{p_1,p_2}=C\cap C_{p_1,p_2}\neq\emptyset$ and this intersection consists of a unique circle, $S$, that is isotopic to $S_{p_1}$ and $S_{p_2}$.
\end{lema}

We will provide an outline of the proof, (see also ~\cite{GZ}),
since some of the details are needed in the proof of
Proposition~\ref{p.basic}. The statement of Theorem 6.1 in~\cite{GZ}
assumes the diffeomorphism is structurally stable.  Structurally
stable diffeomorphisms have the strong transversality property which
implies that for all $x\in\Lambda_0$ and $z\in\Lambda$ that
\begin{equation}\label{e.transverse}
\begin{array}{llll}
W^s(x)\cap W^u(z)=W^s(x)\pitchfork W^u(z),\textrm{ and }\\
W^u(x)\cap W^s(z)=W^u(x)\pitchfork W^s(z).
\end{array}
\end{equation}
This is the property used in the proof in~\cite{GZ}.  However, if $f$ is quasi-Anosov and $\Lambda_0$ and $\Lambda$ are codimension-one, then for all $x\in\Lambda_0$ and $z\in\Lambda$ property
(\ref{e.transverse}) holds
by the quasi-transversal property of quasi-Anosov diffeomorphisms.

First, we use Lemma~\ref{l.intersect} to show there is a component.  Next, we use transversality to show that every component $C\cap C_{p_1,p_2}$ is a circle, $S$, that is isotopic to $S_{p_1}$ and $S_{p_2}$.  So in fact each component divides $C_{p_1,p_2}$ into two cylinders.

Next, let $B_S\subset W^u(z)$ be a minimal disk bounded by $S$.  Since $B_S$ is minimal it follows that $B_S\cap D_{p_1,p_2}=\emptyset$.  Then there are two possibilities.
\begin{enumerate}
\item No $B_S$ contains $z$.
\item Some $B_S$ contains $z$.
\end{enumerate}
It is shown that case (1) can not occur.  For case (2) since $f(C_{p_1,p_2})\subset D_{p_1,p_2}$ we know $f(S)\cap S=\emptyset$ and $S$
is inside $f(S)$ in $W^u(z)$.  It then follows that $S$ and $f(S)$ bound a closed annulus in $W^u(z)$ which is a fundamental domain
of $W^u(z)$ contained in $D_{p_1,p_2}$.  Thus the intersection of $W^u(z)$ and $C_{p_1,p_2}$ is a unique circle.

\noindent{\bf Proof of Proposition~\ref{p.basic}.} To simplify the
argument we first assume the attractors are orientable.  Let us
suppose that $\Lambda$ is a basic set that is a codimension-one and
$\Lambda$ is not a hyperbolic attractor.  Since periodic points are
dense in $\Lambda$ and $\Lambda_0$ we may assume there exist
periodic points $x\in\Lambda$ and $x_0\in\Lambda_0$ such that
$W^s(x_0)\cap W^u(x)\neq\emptyset$.  Let $y\in W^s(x_0)\cap W^u(x)$.
Then from the previous Lemma $y\in (y_1,y_2)^s_{\emptyset}$ for
$y_1$ and $y_2$ in the unstable manifolds of associated boundary
points $p_1$ and $p_2$, respectively.\par
Let $S_{p_1,p_2}$ be a characteristic sphere for $p_1$ and $p_2$
such that $(y_1,y_2)^s_{\emptyset}\subset C_{p_1,p_2}$, so
$C_{p_1,p_2}\cap W^u(x)\neq\emptyset$. From the previous lemma we
know that $W^u(x)\cap D_{p_1,p_2}$ is a unique component $C$.
Furthermore, there is a fundamental domain of $W^u(x)$ contained in
$D_{p_1,p_2}\subset W^s(\Lambda_0)$.  The invariance of
$W^s(\Lambda)$ implies that $W^u(x)-x\subset W^s(\Lambda_0)$. Hence,
$(W^u(x)-x)\cap \Lambda=\emptyset$.  Since $W^u(x)\cap W^s(x)$ is
dense in a component of $\Lambda$, given by the Spectral
Decomposition Theorem, we know that $\Lambda$ is trivial and
consists of the orbit of $x$.
%
%
In this way, the result follows for orientable attractors
$\Lambda_0$. \par
We now suppose that $\Lambda$ is a codimension-one
basic set and $\Lambda_0$ is a codimension-one {\em non-orientable}
attractor where $\Lambda\rightarrow\Lambda_0$.  This implies that
$M$ is non-orientable from~\cite{plykin2}.  Let $\bar{M}$ be an
orientable manifold and $\pi:\bar{M}\rightarrow M$ is a
(non-branched) double covering of $M$.  Then there exists a
diffeomorphism $\bar{f}$ of $\bar{M}$ that covers $f$.  Furthermore,
$\bar{M}$ contains a hyperbolic orientable codimension-one expanding
attractor $\bar{\Lambda}_0$ such that
$\bar{\Lambda}_0\subset\pi^{-1}(\Lambda_0)$.  The result now follows
from the previous argument by lifting $\Lambda$. $\Box$


\section{Proof of Theorem
\ref{thm.kneser.decomposition}}\label{section thm1.1}
Let us recall some basic definitions and results, which can be found
in \cite{milnor}. The {\de connected sum} of two 3-manifolds is
obtained by removing the interior of a 3-cell from each 3-manifold,
and then matching the resulting boundaries, using an orientation
reversing homeomorphism. The connected sum of $M$ and $M'$ is
denoted $M\#M'$. In order to {\de add a handle} to a connected
3-manifold $M$, one removes the interior of two disjoint 3-cells
from $M$, and matches the resulting boundaries under an orientation
reversing homeomorphism. If one adds a handle to $M$, one obtains a
manifold isomorphic to $M\#S^2\times S^1$. Note that $M\#S^3=M$.\par
A manifold $M\ne S^3$ is {\de prime} if $M=M_1\#M_2$ implies
$M_1=S^3$ or $M_2=S^3$. We have the following Unique Decomposition
Theorem (see also \cite{kneser}):
\begin{teo}[Milnor \cite{milnor}] Every 3-manifold $M\ne S^3$ can be
written as a finite connected sum:
$$M=M_1\#\dots\# M_k$$
where each $M_i$ is prime, $i=1,\dots, k$, and is unique up to order
and isomorphisms.
\end{teo}
Note that the handles $S^2\times S^1$ are prime manifolds. The torus
${\mathbb T}^3$ is also prime.\par Now, let us prove Theorem
\ref{thm.kneser.decomposition}. Let $M$ be a hyperbolic invariant
set for a diffeomorphism $f$ such that $M$ is a 3-manifold. Then
$f|M$ is a quasi-Anosov diffeomorphism \cite{manhe}. Let us first
assume that $M$ is orientable. Theorem \ref{thm.medvedev.zhuzhoma}
implies that all attractors and repellers of $f|M$ are orientable.
Then Theorems \ref{homotopy.torus} and \ref{teo.plykin} imply that
the basin of attraction/repulsion of each attractor/repeller is
homeomorphic to a finitely punctured torus. Also, that $f$
restricted to each basin is topologically equivalent to a
DA-diffeomorphism. Let us first consider the simplest case of a non-Anosov quasi-Anosov diffeomorphism: one with just one attractor and
one repeller. Let us furthermore suppose that the basin of
attraction of the attractor is homeomorphic to a torus minus one
point. If one takes a ball centered at that (repelling) point, and
cuts the pre-image of that ball under the DA-diffeomorphism, one
obtains a fundamental domain, which is homeomorphic, under the
conjugacy, to a fundamental domain $D$ of $f|M$, met just once by
the orbit of every point not in the attractor nor in the repeller of
$f|M$. By connectedness, the basin of repulsion of the repeller must
be also homeomorphic to a torus minus one (attracting) point $A$.
And $f$ restricted to this basin is also conjugated to another
DA-diffeomorphism. The image of $D$ under this second conjugacy
consists of two spheres $S^3$, each bounding a ball containing $A$
in its interior. So, the fact that $D$ is a fundamental domain
implies that $M$ is the connected sum of two tori, just like in the
Franks-Robinson example \cite{frob}.\par Let us suppose now that
we have an attractor and a repeller, but that the basin of
attraction is homeomorphic to a torus minus $k$ points, with $k\geq
2$. One can assume that the $k$ points
are fixed under the DA-diffeomorphism by taking a sufficiently high iterate of the diffeomorphism. A connectedness argument
shows that the basin of repulsion of the repeller is also
homeomorphic to a torus minus $k$ points. Now, the previous
procedure shows that $M$ is obtained by removing $k$ 3-cells from
each torus and matching the resulting boundaries, using an
orientation reversing homeomorphism. Observe that this implies that
$$M=T_1\# T_2\# H_1\#\dots\# H_{k-1}$$
where $T_1,T_2$ are tori and the $H_i$ are handles $S^2\times S^1$.
Indeed, instead of removing simultaneously the $k$ 3-cells of each
torus, one can only remove one 3 cell from each torus and glue along
their boundaries by $f$, which reverses orientation. In this way one
obtains the connected sum of two tori. The rest of the procedure
consists in repeating $(k-1)$ times the operation: cutting two
disjoint 3-cells of this connected sum and matching the resulting
boundaries by a reversing orientation homeomorphism. This is the
same as adding a handle.\par
Now, in fact there is nothing special in having just one attractor
and one repeller. In case there are more attractors or repellers,
one proceeds inductively as in the previous cases until one obtains
a finite connected sum of tori and handles.\par
Let us also consider the case where the non-wandering set has a
basic set $\Lambda$ that is a periodic orbit, we may assume
$\st(\Lambda)=\st(\Lambda_0)$, and that $\Lambda\to\Lambda_0$,
where $\Lambda_0$ is a connected codimension-one attractor.  Let
$p$ and $q$ be associated boundary periodic points and
$P_1,...,P_k$ be the set of all periodic points with
$\mathrm{st}(P_i)=1$ and $$W^u(P_i)\cap \bigcup_{x\in
W^u(p)-\{p\}}[x,\phi_{p,q}(x)]^s_{\emptyset}\neq\emptyset$$ for
all $1\leq i\leq k$.  Let $x\in W^u(p)-\{p\}$ and
$$\theta_x:[x,\phi_{p,q}(x)]^s_{\emptyset}\rightarrow [0,1]$$ be a
homeomorphism.  Denote $P_i^x$ to be the point of intersection
between $[x,\phi_{p,q}(x)]^s_{\emptyset}$ and $W^u(P_i)$.  By
reordering the points, if necessary, assume that
$\theta_x(P_i^x)<\theta_x(P_j^x)$ for $i<j$.  Then for any other
$x'\in W^u(x')$ and $\theta_{x'}$ defined similarly we have
$\theta_{x'}(P_i^{x'})<\theta_{x'}(P_j^{x'})$ for $i<j$ since
$W^u(P_i)$ is codimension-one.

For the rest of the construction we assume, with no loss of
generality, that $p,q, P_1,...,P_k$ are all fixed points. Define
the set
$$D=\bigcup_{x\in W^u(p)-\{p\}}[x,\phi_{p,q}(x)]^s_{\emptyset}\cup
W^s_{\emptyset}(p)\cup W^s_{\emptyset}(q)\cup W^s(P_1)\cup \cdots
\cup W^s(P_k).$$

Following the construction in the proof of Theorem 1 of
~\cite{plykin1} we can extend the diffeomorphism $f|_D$ to a
homeomorphism $\bar{f}$ on the compactification $\bar{D}=D\cup
\alpha_1\cup \cdots \alpha_{k+1}$ where each $\alpha_i$ is a
repelling fixed point for $\bar{f}$. Fix $\epsilon>0$ sufficiently
small and let
$$B=\bigcup_{x\in W^u_{\epsilon}(p)-\{p\}}[x,\phi_{p,q}(x)]^s_{\emptyset}\cup
W^s_{\emptyset}(p)\cup W^s_{\emptyset}(q)\cup (\bigcup_{i=1}^k
W^s(P_i))\cup (\bigcup_{i=1}^{k+1}\alpha_i).$$  Let
$B_{\epsilon}(0)$ be the ball of size $\epsilon$ centered at the
origin in $\mathbb{R}^{n-1}$. Then we can define a homeomorphism
$F: B_{\epsilon}(0)\times[0,1]\rightarrow B$ so that
\begin{itemize}
\item $F(0,0)=p$,
\item $F(0,1)=q$,
\item $F(0,\frac{2i}{2k+1})=P_i$ for $1\leq i\leq k$,
\item $F(0,\frac{2i+1}{2k+1})=\alpha_{i+1}$ for $0\leq i\leq k$,
\item $F(x,t)\in [F(x,0),\phi_{p,q}(F(x,0))]^s_{\emptyset}$, and
\item $F(x,0)\in W^u_{\epsilon}(p)$.
\end{itemize}
Furthermore, we can extend the unstable manifolds of the points in
$p\cup P_1\cup \cdots \cup P_k$ to a codimension-one fibration of
$\bar{D}$ and extend the stable foliation to a fibration of
$\bar{D}$ with one-dimensional fibers.  Similar to Corollary 7.2
in~\cite{GZ}, there exists a compact arc $a_{pq}\subset\bar{D}$
with no self intersections such that
$$a_{pq}=p\cup W^s_{\emptyset}(p)\cup \alpha_1 \cup
W^s(P_1)\cup \alpha_2 \cup \cdots \cup P_k\cup \alpha_{k+1}\cup
W^s_{\emptyset}(q)\cup q.$$

Let $\mathcal{P}$ be the set of saddle periodic points of stable
index one that intersect $W^s(\Lambda)$.  Following the above
construction we compactify $W^s(\Lambda)\cup W^s(\mathcal{P})$ to
a set $\overline{W^s(\Lambda)}$ and extend the diffeomorphism $f$
on $W^s(\Lambda)\cup W^s(\mathcal{P})$ to a homeomorphism
$\bar{f}$ of $\overline{W^s(\Lambda)}$.  Where
$\overline{W^s(\Lambda)}=W^s(\Lambda)\cup W^s(\mathcal{P})\cup
\mathcal{A}$ where $\mathcal{A}$ consists of a set of repelling
periodic points for $\bar{f}$.  The proof of Theorem 7.1
in~\cite{GZ} extends to $\overline{W^s(\Lambda)}$ to show that
$\overline{W^s(\Lambda)}$ is homeomorphic to $\mathbb{T}^3$.

In the case where there is at least one non-orientable attractor
or repeller, then $M$ is non-orientable. Theorem \ref{teo.plykin}
implies that the basin of attraction or repulsion of this
attractor or repeller is homeomorphic to a torus quotiented by an
involution minus $k$ points. $f$ is doubly covered by a
DA-diffeomorphism in this set, and the procedure of removing cells
and matching the corresponding boundaries follows as in the
previous cases, whence one obtains that
$$M=\tilde T_1\#\dots\#\tilde T_n\# H_1\#\dots\# H_k$$
where the $\tilde T_j=\T^3|\theta_j$ are tori quotiented by
involutions $\theta_j$ (possibly the identity), and the $H_i$ are
handles.

\section{An example of a quasi-Anosov diffeomorphism with a basic set that is a periodic orbit}
Both the example by Franks-Robinson \cite{frob}, and the example
by Medvedev-Zhuzhoma \cite{medvedevzhuzhoma} are quasi-Anosov
diffeomorphisms whose non-wandering set consists exclusively of
one codimension-one expanding attractor and one shrinking repeller.\par%
Let us construct an example of a quasi-Anosov diffeomorphism with a
basic set consisting of a periodic orbit.\par%
Let $A$ be a linear hyperbolic diffeomorphism of a 3-torus $T_1$
having at least one fixed point, and such that the stable
dimension is $1$. Make a deformation around a fixed point, in
order to obtain a DA-diffeomorphism with a repelling fixed point
and a codimension-one expanding attractor. See details in Section
\ref{section.thm1.2}. This new diffeomorphism $h$ preserves the
original stable foliation. Now, make a new deformation, also
preserving the original stable foliation, such that the repelling
fixed point turns into a saddle, and two repelling fixed points
appear on its stable manifold, locally
separated by the unstable manifold of the saddle point.\par%
Like in Franks-Robinson example, cut two 3-balls $B_2$ and $B_3$
containing, respectively, the repelling fixed points of $T_1$.
Now, take two 3-tori $T_2$ and $T_3$, and consider the dynamics of
$h^{-1}$ on each one of them. Cut two 3-balls $B'_2$ and $B'_3$,
each containing the attracting fixed points of $T_2$ and $T_3$
respectively. Glue carefully $T_1$ and $T_j$ along the boundary of
$B_j$ and $B'_j$, $j=2,3$, by means of a orientation reversing
homeomorphisms as described in section \S\ref{section thm1.1}. One
obtains an Axiom A diffeomorphism of a manifold which is the
connected sum of $T_1$, $T_2$ and $T_3$: $T_1\# T_2\# T_3$. The
non-wandering set of this diffeomorphism consists of one
codimension-one expanding attractor, two codimension-one shrinking
repellers, and a hyperbolic fixed point.  Proceeding as in
\cite{frob} one perturbs the diffeomorphism producing a twist in
the regions where the surgery was performed. Since this
perturbation is local, it does not affect the hyperbolic behavior
of the non-wandering set. In this way one obtains an Axiom A
diffeomorphism satisfying the quasi-transversality condition, and
with the above mentioned non-wandering set.

\section{Partially hyperbolic quasi-Anosov diffeomorphisms
\label{section.partially.hyperbolic}}
In this section, we study quasi-Anosov diffeomorphisms in the
presence of partial hyperbolicity. A diffeomorphism $f$ is called
{\de partially hyperbolic} if there exists an invariant splitting of
the tangent bundle $TM=E^s\oplus E^c\oplus E^u$ such that all unit
vectors $v^\sigma\in E_x^\sigma$, $\sigma=s,c,u$ satisfy:
$$|Tf_x v^s|<|Tf_x v^c|<|Tf_x v^u|\quad\mbox{and}\quad |Tf_x v^s|<1<|Tf_x v^u|$$
It is a known fact that there are unique invariant foliations
$\mathcal W^s$ and $\mathcal W^u$ that are everywhere tangent,
respectively, to $E^s$ and $E^u$ (see, for instance \cite{hps}).
However, $E^c$ is not integrable in general. It is an open problem
if it is integrable in the case $\dim E^c=1$. Here, we shall
consider an a priori mild hypothesis: either $E^{cs}$ or $E^{cu}$
integrates to a (codimension-one) foliation. \newline\par %
Let us prove Theorem \ref{ph.qad}. Assume that $E^{cu}$ integrates
to a foliation $\mathcal F$. We may suppose that the manifold $M$ is
not the 3-torus, for otherwise the result is immediate. Then, $M$ is
not {\de irreducible} (that is, there is an embedded sphere $S^2$
not bounding any $3$-dimensional ball). Indeed, irreducible
manifolds are prime, and Theorem \ref{thm.kneser.decomposition}
implies that $M$ is not prime unless it is the 3-torus. Now, Theorem
C.2., p 45 of Roussarie \cite{roussarie} implies that a
codimension-one foliation $\mathcal F$ of a manifold that is not
irreducible has a compact leaf. This compact leaf $T$ must be
homeomorphic to a 2-torus, since the strong unstable foliation has
no singularities and does not contain closed leaves.\par%
 Take $p\in M$ a periodic point in $M$, of period $k$, such that the stable
 manifold of $p$ hits $T$. This implies that the set of leaves
 $f^{kn}(T)$ ($f$ preserves the foliation) accumulates in $p$.
 Now, Haefliger \cite{haefliger} says that the set of points lying in a compact leaf is
 compact. In particular $p\in M$, accumulating point of compact
 leaves, must be in a compact leaf $T_0$. It follows that $T_0$ is $f^k$-invariant.
 Moreover, by the above argument, $T_0$ must be a 2-torus.
 Now, the local stable manifolds of $T_0$ form an open set $U$ satisfying
$f^k(\overline{U})\subset U$. This implies that $T_0$ is a basic set
that is an attractor, but this contradicts Theorem
\ref{teo.quasi.Anosov.dim3}.(2).

\end{document}